\documentclass[12pt]{amsart}
\usepackage{ amscd}
\def\Z {\mathbb{Z}}
\def\Q {\mathbb{Q}}

\def \co{\colon\!}

\def \ftnote{\let\thefootnote\relax\footnotetext}

\def\cat{{\mathop\mathrm{cat}\,}}
\def\TC{{\mathop\mathrm{TC}\,}}

\usepackage{color}

\newtheorem{theorem}{Theorem}
\newtheorem{problem}[theorem]{Problem}
\newtheorem{lemma}[theorem]{Lemma}
\newtheorem{corollary}[theorem]{Corollary}
\newtheorem{proposition}[theorem]{Proposition}

\theoremstyle{remark}

\title[]{The topological complexity of the free product}
\author{Alexander Dranishnikov, Rustam Sadykov}
\date{}

\begin{document}
\address{A. Dranishnikov, Department of Mathematics, University
of Florida, 358 Little Hall, Gainesville, FL 32611-8105, USA}
\email{dranish@math.ufl.edu}

\address{R. Sadykov, Mathematics Department, Kansas State University,
138 Cardwell Hal,l Manhattan, KS 66506}
\email{sadykov@math.ksu.edu}
\maketitle

\begin{abstract} 
We prove the formula
\[
   \TC(G\ast H)=\max\{\TC(G), \TC(H), cd(G\times H)\}
\]
for the topological complexity of the free product of groups with cohomological dimension $\ge 3$.
\end{abstract}

\section{Introduction}

Let $X$ be a topological space. The {\emph{Lusternik-Schnirelmann category} (LS-category)  $\cat X$ of  $X$ is the least number $n$ such that there is a covering  $\{U_i\}$ of $X$ by $n+1$ open sets $U_i$ contractible in $X$ to a point. The LS-category has a number of interesting applications ~\cite{CLOT}. 
A \emph{motion planning algorithm} over an open subset $U_i\subset X\times X$ is a continuous map $U_i\to X^{[0,1]}$ that takes a pair $(x,y)$ to a path $s$ with end points $s(0)=x$ and $s(1)=y$.
The \emph{topological complexity} $\TC(X)$ of  $X$ is the least number $n$  such that there is a covering $\{U_i\}$ of $X\times X$ by $n+1$ open sets over which there are motion planning algorithms~\cite{Fa1}. We note that in this paper we consider the reduced LS-category and the reduced topological complexity.

Both $\cat X$ and $\TC(X)$ are homotopy invariant. Thus one can define the LS-category and the topological complexity of a discrete group $G$ by setting $\cat G=\cat BG$ and $\TC(G)=\TC(BG)$ where $BG=K(G,1)$ is a classifying spaces for $G$. It is well-known that $\cat G$ does not give a new invariant~\cite{EG},\cite{St},\cite{Sw}.
\begin{theorem}[Eilenberg-Ganea]\label{th:1}
For all groups $\cat G=cd(G)$. 
\end{theorem}
Here $cd(G)$
 is the cohomological dimension of $G$. By the Eilenberg-Ganea theorem $cd(G)$ equals the geometric dimension 
$gdim(G)=\min\{\dim BG\}$ of $G$ provided $cd(G)\ne 2$~\cite{Br}. {\em The Eilenberg-Ganea conjecture} states
that the equality $cd(G)=gdim(G)$ holds for all groups. A potential  counterexample to Eilenberg-Ganea conjecture should have $cd(G)=2$ and $gdim(G)=3$~\cite{Br}. 

The topological complexity of a group is secluded in the range $cd(G)\le \TC(G)\le 2cd(G)$ and any value between can be taken~\cite{Ru}.
Computation of topological complexity of a group is a great challenge.  Since $\TC(G)=\infty$
for groups with torsion, this invariant makes sense only for torsion free groups. It was computed only for few classes of groups.
Thus, the topological complexity of free abelian group equals the rank. The topological complexity of nilpotent groups is estimated in~\cite{Gr}.
Computation of topological complexity of surface groups orientable and non-orientable was only recently completed~\cite{Fa1},\cite{Dr3},\cite{CV}.
Even more recent is the computation of $\TC$ of hyperbolic groups~\cite{FM}.

In this paper we present a formula for topological complexity of free product of groups.
\begin{theorem}\label{th:2}
If both  groups  $G$ and $H$ are not counterexamples to the Eilenberg-Ganea conjecture, then
\[
   \TC(G\ast H)=\max\{\TC(G), \TC(H), cd(G\times H)\}.
\]
\end{theorem}
In view of the inequality $\TC(X)\le\cat(X\times X)$ we obtain the following:
\begin{corollary} If $G$ is not a counterexample to the Eilenberg-Ganea conjecture, then
\[
   \TC(G\ast G)=cd(G\times G).
\]
\end{corollary}
\begin{corollary}
$$\TC(\Z^m\ast\Z^n)=m+n.$$
\end{corollary}
\begin{corollary}[Yu. Rudyak]
For any $n$ and $k$ with $n\le k\le 2n$ there is a group $G$ having $cd(G)=n$ and $\TC(G)=k$
\end{corollary}
\begin{proof}
Take $G=\Z^m\ast\Z^n$ for $m=k-n$.
\end{proof}

REMARK 1. If $G$ (or $H$) is a counterexample to the Eilenberg-Ganea conjecture having $\TC(G)=4$ ($\TC(H)=4$), then still Theorem~\ref{th:2} holds true. 

REMARK 2. Theorem~\ref{th:2} holds for groups $G$ and $H$ with $cd(G)\ge cd(H)\ge 2$ and $cd(G\times H)\ge cd(G)+2$.
We beleive that the latter inequality holds true for all geometrically finite groups $H$ with $cd(H)\ge 2$.
We note that raising dimension by 2 of the Cartesian product with a nonfree group is the maximal possible in view of an example in~\cite{Dr2} of geometrically finite groups $G$ and $H$  with $cd(G)=cd(H)=3$ and  $cd(G\times H)\le 5$. The groups $G$ and $H$ in \cite{Dr2} are finite index subgroups of right-angled Coxeter groups
constructed by an appropriate choice of the nerves. Also we note that in the class of infinitely generated groups there is a counterexample: $cd(\Q)=2$ while $cd(\Q\times\Q)=3$.
We are thankful to David Recio-Mitter for the latter remark.

\section{Topological complexity of wedge}

It was proved in~\cite{Dr1}, Theorem 3.6,  that 
\[
   \max\{\TC X, \TC Y, \cat (X\times Y)\}\le  \TC(X\vee Y).
\]
  We show that under certain conditions the lower estimate is exact. 

\begin{theorem}\label{th:3}  Let  $d=\max\{\dim X,\dim Y\}$ for connected CW-complexes $X$ and $Y$. Suppose
that $\max\{\TC X,\TC Y,\cat(X\times Y\}\ge d+1$. Then
\[
\TC(X\vee Y)=   \max\{\TC X, \TC Y, \cat (X\times Y)\}.
\]
\end{theorem}

We note that in many cases  this is an improvement of the upper bound from~\cite{Dr1}
\[
   \TC(X\vee Y)\le\TC^M X+ \TC^M Y
\]
where $\TC^M$ is the monoidal topological complexity (see~\cite{I-S}). We recall that $\TC X\le \TC^M\le \TC X+1$.
In view of the fact that $\cat (X\times Y)\le \cat X+\cat Y$, Theorem~\ref{th:3} implies the inequality conjectured in~\cite{Fa2}, ~\cite[Remark 3.7]{Dr1}
\[
 \TC(X\vee Y)\le \max\{\TC X, \TC Y, \cat X+\cat Y\}
\]
 under the hypothesis of Theorem~\ref{th:3}. 

We extend Theorem~\ref{th:3} in two directions.
First observation is that the condition of  Theorem~\ref{th:3} can be relaxed in case of $r$-connected complexes.
\begin{theorem}\label{th:4}  Let  $d=\max\{dim X,\dim Y\}$ for $r$-connected CW-complexes $X$ and $Y$. Suppose
that  $\max\{\TC X,\TC Y,\cat(X\times Y\}\ge \frac{d+1}{r+1}$. Then
\[
\TC(X\vee Y)=   \max\{\TC X, \TC Y, \cat (X\times Y)\}.
\]
\end{theorem}
We note that the condition $\TC X\ge (\dim X+1)/(r+1)$ for $r$-connected  complexes has been seen before in the TC theory. Thus,
the first author showed~\cite{Dr1} that the topological complexity $\TC X$ coincides with the monoidal topological complexity $\TC^M X$ 
for $r$-connected $X$ under the assumption that $ \TC X\ge (\dim X+1)/(r+1)$. 

Then we extend the upper bound of Theorem~\ref{th:3} from the wedge of two spaces to the union.
\begin{theorem}\label{th:5}
Let  $d=\max\{dim (X\times C),\dim (Y\times C)\}$ for connected CW-complexes $X$ and $Y$ where $C=X\cap Y$. Suppose
that $\max\{\TC X,\TC Y,\cat(X\times Y\}\ge d+1$. Then
\[
\TC(X\cup Y)\le   \max\{\TC X, \TC Y, \cat (X\times Y)\}.
\]
\end{theorem}

\section{Preliminaries}

\subsection{Fiberwise join}
The \emph{join} $X_0*X_1*\cdots*X_n$ of topological spaces consists of formal linear combinations $t_0x_0+\cdots +t_nx_n$ of points $x_i\in X_i$ with non-negative coefficients $t_i$ that satisfy the condition $\sum t_i=1$. The \emph{fiberwise join} of the total spaces $X_0, ..., X_n$ of fibrations $f_i\co X_i\to Y$  is defined to be the topological space \[
    X_0*_YX_1*_Y\cdots *_YX_n=\{\ t_0x_0+\cdots +t_nx_n\in X_0*\cdots *X_n\ |\ f_0(x_0)=\cdots =f_n(x_n)\ \}.
\]
It is fibered over $Y$ by means of the map, called the \emph{fiberwise join of fibrations}, 
defined by taking a point $t_0x_0+\cdots +t_nx_n$ to $f_i(x_i)$ for any $i$. As the name `fiberwise join' suggests, the fiber of the fiberwise join of fibrations is given by the join of fibers of fibrations. 

In all our applications, the spaces $X_i$ coincide with a space $X$ and all fibrations $f_i$ coincide with a fibration $f$ over a space $Y$. In this case the fiberwise join and the fiberwise join of fibrations are denoted by  $*^{n+1}_YX$ and  $*_Y^{n+1}f$ respectively. 

Let $G^0_X$ be the space of paths $s\co [0,1]\to X$ issued from the base point of $X$. It is fibered over $X$ by means of the map $s\mapsto s(1)$. By definition, the \emph{$n$-th Ganea space} of $X$ is the fiberwise join $G^n_X=*^{n+1}G^0_X$.  By the Schwarz theorem~\cite{Sch},  $\cat(X)\le n$ if and only if the $n$-th Ganea fibration $p^n_X:G^n_X\to X$ admits a section.

Similarly, let $PX$ denote the space of paths $s\co [0, 1]\to X$. It is fibered over $X\times X$ by means of the map $s\mapsto (s(0), s(1))$. Let $\Delta^n_X$ denote the fiberwise join $*^{n+1}PX$. Then, by the Schwarz theorem, $\TC X\le n$ if and only if $\Delta^n_X$ admits a section. 

\subsection{The Berstein-Schwarz class.} Let $\pi$ be a discrete group and  $A$ be a $\pi$-module. By $H^*(\pi,A)$ we denote the cohomology of the group $\pi$ with coefficients in $A$ and by
$H^*(X;A)$ we denote the cohomology of a space $X$ with the twisted coefficients defined by $A$. Here
we assume  $\pi_1(X)=\pi$.

The Berstein-Schwarz class of a group $\pi$ is a certain cohomology class $\beta_{\pi}\in H^1(\pi,I(\pi))$ where $I(\pi)$ is the augmentation ideal of the group ring $\Z\pi$~\cite{Ber},\cite{DR},\cite{Sch}. 
It is defined as the first obstruction to a section over $B\pi=K(\pi,1)$ for the universal covering $p:E\pi\to B\pi$. 
Equivalently $\beta_{\pi}$ can be defined as follows. Let
$$
0\to I(\pi)\to\Z\pi\stackrel{\epsilon}\to\Z\to 0
$$ 
be a short exact sequence of coefficients where $\epsilon$ is the augmentation homomorphism.
Then $\beta_{\pi}=\delta(1)$ equals the image of the generator $1\in H^0(\pi;\Z)=\Z$ under the connecting homomorphism 
$\delta:H^0(\pi;\Z)\to H^1(\pi;I(\pi))$ in the coefficient long exact sequence.

\begin{theorem}[Universality]~\cite[Proposition 2.2]{DR},\cite{Sch}
For any $\pi$-module $L$ and any cohomology class $\alpha\in H^k(\pi,L)$ there is a homomorphism of $\pi$-modules $I(\pi)^k\to L$ such that the induced homomorphism for cohomology takes $(\beta_{\pi})^k\in H^k(\pi;I(\pi)^k)$ to $\alpha$  where $I(\pi)^k=I(\pi)\otimes\dots\otimes I(\pi)$ and $(\beta_{\pi})^k=\beta_{\pi}\smile\dots\smile\beta_{\pi}$.
\end{theorem}
\begin{corollary}[\cite{Sch}]\label{ganea-obstr}
The class  $(\beta_{\pi})^{n}$ is the primary obstruction to a section of $p_{B\pi}^{n-1}:G^n_{B\pi}\to B\pi$.
\end{corollary}
\begin{corollary}\label{cd} 
For any group $\pi$ its cohomological dimension can be expressed as follows:
$$
cd(\pi)=\max\{n\mid (\beta_{\pi})^n\ne 0\}.
$$
\end{corollary}

\subsection{Pasting sections}
We recall that a map $p:E\to B$  satisfies the {\em Homotopy Lifting Property for a pair} $(X,A)$ if for any
homotopy $H:X\times I\to B$ with a lift $H':A\times I\to E$ of the restriction $H|_{A\times I}$
and a lift $H_0$ of $H|_{X\times0}$ which agrees with $H'$, there is a lift $\bar H:X\times I\to E$
of $H$ which agrees with $H_0$ and $H'$. We recall that a pair of spaces $(X,A)$ is called an NDR pair if $A$ is a deformation retract of a neighborhood in $X$. In particular, every CW complex pair is an NDR pair.  It is well-known~\cite{TD}, Corollary 5.5.3 that any Hurewicz fibration $p:E\to B$  satisfies the Homotopy Lifting Property for NDR pairs $(X,A)$.
\begin{lemma}\label{paste}
Let $p:E\to B$ be a Hurewicz fibration over a CW complex $B$ with $(n-1)$-connected fiber $F$ where $B=X\cup Y$ is presented as the union of subcomplexes with $n$-dimensional intersection  $C=X\cap Y$, $\dim C =n$, such that $H^n(C;\mathcal F)=0$
for any local coefficients. Suppose that there are sections of $p$ over $X$ and $Y$. Then $p$ admits a section $s:B\to E$. 
\end{lemma}
\begin{proof} Let $s_X:X\to E$ and $s_Y:Y\to E$ be the sections.
First we show that there is a fiberwise homotopy between the restrictions $s_X$ and $s_Y$ to $C$.
We note that construction of such homotopy is a relative lifting problem
$$
\begin{CD}
C\times\partial I @>s_X\cup s_Y>> E\\
@VVV @Vp|VV\\
C\times I @>\pi>> C\\
\end{CD}
$$
for the projection map $\pi$.
Since the fiber is $(n-1)$-connected, a lift of $\pi$ exists on the $n$-skeleton $(C\times I)^{(n)}=(C^{(n-1)}\times I)\cup (C\times\partial I)$. The obstruction to extend it to the $(n+1)$-skeleton lives in the cohomology group $H^{n+1}(C\times I,C\times\partial I;\mathcal F)$ for local coefficients defined by
$\pi_n(F)$. The exact sequence of pair together with the acyclicity of $C$ in dimensions $n$ and $n+1$ imply that
$H^{n+1}(C\times I,C\times\partial I;\mathcal F)=0$. This proves the existence of a fiberwise homotopy
$H:C\times I\to E$ between the sections $s_X$ and $s_Y$ restricted to $C$.
In view of the Homotopy Lifting Property for NDR pairs the fiberwise homotopy $H:C\times I\to E$ of the restriction $s_X|_C$ can be extended to a fiberwise homotopy $\bar H:X\times I\to E$ of $s_X$. Then the section $s'_X:X\to E$ defined as $s'_X(x)=H(x,1)$ agrees with $s_Y$ on $C$. Hence the union $s'_X\cup s_Y$ defines a section $s:X\cup Y\to E$.  
\end{proof}

\section{Proof of Theorem~\ref{th:2}}

{\em Proof of Theorem~\ref{th:1}.} Corollary~\ref{cd} and the cup-length lower bound for the LS-category imply that $\cat G\ge cd(G)$.
The dimension upper bound for the LS-category  completes the proof for the groups with $cd(G)=gdim(G)$. Now suppose that $cd(G)=2$ and $\dim BG=3$.
Note that the Ganea-Schwarz fibration $G^2_{BG}\to BG$ has simply connected fiber. Thus, there is a section $s':BG^{(2)}\to G^2_{BG}$. The primary 
(and the only) obstruction to define a section $s:BG\to G^2_{BG}$ lives in the group $H^3(BG;\mathcal F)$ which is trivial in view of the equality $cd(G)=2$.\qed

\begin{proposition}\label{+one}
For all groups $cd(G\times H)\ge cd(G)+1$.
\end{proposition}
\begin{proof}
We may assume that the groups have finite cohomological dimension. In particular, $H$ is torsion free. Hence it contains a copy of integers.
Since the cohomological dimension of a subgroup
does not exceed the cohomological dimension of a group~\cite{Br}, it follows
that $$cd(G\times H)\ge cd(G\times\Z)\ge cd(G)+1.$$
\end{proof}

\

{\em Proof of Theorem 2.} 
Since both $G$ and $H$ are not counterexamples to the Eilenberg-Ganea conjecture, there are classifying spaces $BG$
and $BH$ with $\dim BG=cd(G)$ and $\dim BH=cd(H)$.

In view of  Theorem~\ref{th:1} and Proposition~\ref{+one} we obtain
$$\cat(BG\times BH)=\cat(G\times H)=cd(G\times H)\ge \max\{cd(G),cd(H)\}+1.$$ 
Thus, by the Eilenberg-Ganea theorem $$\cat(BG\times BH)\ge\max\{\dim(BG),\dim(BH)\}+1.$$
By Theorem~\ref{th:3} we obtain
$$
\TC(G\ast H)=\TC(BG\vee BH)=\max\{\TC G, \TC H,\cat(BG\times BH)\}=$$
$$\max\{\TC G, \TC H,\cat(B(G\times H))\}=\max\{\TC G, \TC H,cd(G\times H)\}.
$$
\qed

\begin{problem}\label{e-g}
Suppose that $G$ is a counterexample to the Eilenberg-ganea conjecture, i.e. $cd(G)=2$ and $gdim(G)=3$. Does it follow that $\TC(G)=4$?
\end{problem}

\section{Proof of Theorems~\ref{th:3},~\ref{th:4},~\ref{th:5}}

Let $Y$ be a subspace of $X$. Then the inclusion $i\co Y\to X$ gives rise to the map of fibrations $i_\Delta\co \Delta^k_Y\to \Delta^k_X$. In particular, if $\Delta^k_Y$ admits a section over $Y\times Y$, then $\Delta^k_X$ also admits a section over $Y\times Y\subset X\times X$. Suppose, furthermore, that there is a retraction $r\co X\to Y$. It gives rise to  the map of fibrations $r_\Delta\co \Delta^k_X\to \Delta^k_Y$. In particular, the existence of a section of $\Delta^k_X$ implies the existence of a section of the fibration $\Delta^k_Y$. 

Similarly, given two topological spaces $X$ and $Y$, for each $k\ge 0$, there is a fibration $G^k_{X\times Y}$, called the $k$-th Ganea fibration, over $X\times Y$ such that $\cat (X\times Y)\le k$ if and only if $G^k_{X\times Y}$ admits a section. A point in $G^k_{X\times Y}$ over $(x, y)\in X\times Y$ is a formal sum $\sum t_i(g_i, h_i)$ where $g_i$ and $h_i$ are paths from respectively $x$ and $y$ to the distinguished point in $X$ and $Y$ respectively.

Theorem~\ref{th:3} is a partial case of Theorem~\ref{th:4}. So we prove the latter.

\subsubsection*{The lower bound for $\TC (X\vee Y)$} In this subsection we give an alternative proof of the fact that the topological complexity $\TC(X\vee Y)$ of the pointed sum of topological spaces $X$ and $Y$ is bounded below by $\TC X$, $\TC Y$ and $\cat (X\times Y)$ (Theorem 3.6~\cite{Dr1}). 
We note that this lower bound works without any conditions.

Since both $X$ and $Y$ are retracts of $X\vee Y$,  the existence of a section of $\Delta^k_{X\vee Y}$ implies the existence of sections of $\Delta^k_X$ and $\Delta^k_Y$. In other words, the topological complexity of $X\vee Y$ is bounded below by $\TC X$ and $\TC Y$. 

To show that $\TC (X\times Y)$ is also bounded by $\cat (X\times Y)$, consider the map of fibrations 
\[
   p\co   \Delta^k_{X\vee Y}|X\times Y\longrightarrow G^k_{X\times Y}
\]
of the restriction of $\Delta^k_{X\vee Y}$  to $X\times Y\subset (X\vee Y)\times (X\vee Y)$ where $k=\TC (X\vee Y)$. The map $p$ is given by 
\[
    \sum t_i f_i\mapsto \sum t_i (p_1\circ f_i, p_2\circ \bar{f}_i),
\] 
where $p_i$ is the projection of $X\vee Y$ to the $i$-th factor, and $\bar{f}_i$ is the path $f_i$ traversed in the opposite direction. Suppose that the fibration $\Delta^k_{X\vee Y}$ admits a section $s_\vee$. Then the fibration $G^k_{X\times Y}$ also admits a section $s_G$ defined as $s_G(x, y)=p\circ s_\vee(x, y)$.  Thus, $\cat(X\times Y)\le k$. 

\subsubsection*{The upper bound for $\TC (X\vee Y)$} In this subsection we show that $\TC (X\vee Y)$ is bounded above by the maximum of $\TC X$, $\TC Y$ and $\cat (X\times Y)$. Let $k$ be the maximum of $\TC X$, $\TC Y$ and $\cat (X\times Y)$. In particular, the fibrations $\Delta^k_X$, $\Delta^k_Y$ and $G^k_{X\times Y}$ admit sections. We need to show that $\TC (X\vee Y)\le k$, i.e., the fibration $\Delta^k_{X \vee Y}$ admits a section. 

We assume that $X$ and $Y$ are $r$-connected CW-complexes of dimension $\le d$ with $\TC X$ or $\TC Y$ at least $(d+1)/(r+1)$ with $r\ge 0$. Without loss of generality we may assume that $\TC X\ge (d+1)/(r+1)$. We have shown that $k=\TC (X\vee Y)\ge \TC X$. Thus, $k\ge (d+1)/(r+1)$.  

Note that if $Z$ is an $(r-1)$-connected space, then $*^{k+1}Z$ is $(k+(k+1)r-1)$-connected. Indeed, the join $*^{k+1}Z$ of spaces is homotopy equivalent to the reduced join of spaces, which, in its turn, is homeomorphic to the reduced suspension $\Sigma^k(Z\wedge \cdots\wedge Z)$, where the number of pointed factors is $k+1$. We may assume that besides the distinguished point $\Sigma^k(Z\wedge \cdots \wedge Z)$ has no cells in dimensions below $k+(k+1)r$. Therefore it is $(k+(k+1)r-1)$-connected. In particular, if 
 $X$ and $Y$ are $r$-connected, then $\Omega (X\vee Y)$ is $(r-1)$-connected, and $*^{k+1}\Omega (X\vee Y)$ is $k+(k+1)r-1$-connected, 
\[
   k+(k+1)r-1=kr+k+r-1=k(r+1)+r-1\ge (d+1)+r-1\ge d.
\]
Thus, the fiber of $\Delta^k_{X\vee Y}$ is at least $d$-connected. We show that if $\Delta^k_{X\vee Y}$ admits sections over $X\times X$, $X\times Y$, $Y\times X$, and $Y\times X$, then it admits a section over $(X\vee Y)^2$. 
Let $A=(X\times X)\cup (Y\times Y)$ and $B=(X\times Y)\cup (Y\times X)$. Sections over $X\times X$, $X\times Y$, $Y\times X$, and $Y\times X$
can be taken such that they agree at $(x_0,x_0)$ where $x_0$ is the wedge point. Thus, there are sections over $A$ and $B$.
Note that the intersection $A\cap B=X\vee X\vee Y\vee Y$ is $d$-dimensional. Since the fiber of the fibration is $d$-connected, Lemma~\ref{paste} implies that that there is a section over $A\cup B=(X\vee Y)^2$.

In view of the retractions $X\vee Y\to X$ and $X\vee Y\to Y$ and the fact that $\Delta^k_X$ and $\Delta^k_Y$ have sections, the fibration $\Delta^k_{X\vee Y}$ admits sections over $X\times X$ and $Y\times Y$. Let us show that it also admits a section over $X\times Y$; the case of $Y\times X$ is similar. To this end, consider a map of fibrations 
\[
   q\co G^k_{X\times Y} \longrightarrow \Delta^k_{X\vee Y}|X\times Y
\]
that takes  a point $\sum t_i (g_i, h_i)$ in the fiber over $(x, y)$ to the point $\sum t_i (g_i\cdot\bar{h}_i)$ where $\cdot$ stands for taking the concatenation of two paths. Since $G^k_{X\times Y}$ admits a section $s$, the map $q$ of fibrations gives rise to the section $q\circ s$ of $\Delta^k_{X\vee Y}$ over $X\times Y$. This completes the proof of the upper bound for $\TC (X\vee Y)$.  

Theorem~\ref{th:3} can be generalized to the following
\begin{theorem}\label{modified}
Let  $d=\max\{\dim X,\dim Y\}$ for connected CW-complexes $X$ and $Y$. Suppose
that $k=\max\{\TC X,\TC Y,\cat(X\times Y)\}\ge d$ and $H^k(X;\mathcal F)=0=H^k(Y;\mathcal G)$ for all local coefficients. Then
\[
\TC(X\vee Y)=   \max\{\TC X, \TC Y, \cat (X\times Y)\}.
\]
\end{theorem}
\begin{proof} If $k\ge d+1$, the result follows from Theorem~\ref{th:3}. We may assume that $k=d$.
We need to check the inequality $\TC(X\vee Y)\le k$. We use the above proof for $r=0$ to show that $\Delta^k_{X\vee Y}$ has $(d-1)$-connected
fiber. Then $\Delta^k_{X\vee Y}$ admits a section over $(X\vee Y)^2=A\cup B$ by Lemma~\ref{paste}, since the intersection $A\cap B$ is $d$-dimensional and cohomologically acyclic in dimension $d$.
\end{proof}

\subsection{Proof of Theorem~\ref{th:5}} Let $k=\max\{\TC X,\TC Y,\cat(X\times Y)\}$. We show that the fibration
$\Delta^k_{X\cup Y}$ admits a section. Since $\TC X,\TC Y\le k$, it admits a section over $X\times X$ and $Y\times Y$.
Choosing a common base point $c_0$ for $X$ and $Y$ allows us to embed the base point path space
$P_0(X\times Y)$ to the path space $P(X\cup Y)$ by taking a path $(f,g)$ in $X\times Y$ issued
from the base point $(c_0,c_0)$ to the path $\bar fg$. This defines an embedding of the Ganea-Schwarz fibration $G^k_{X\times Y}$ in the fibration $\Delta^k_{X\cup Y}$. The inequality $\cat(X\times Y)\le k$ implies that $G^k_{X\times Y}$ has a section. Therefore, $\Delta^k_{X\cup Y}$ has a section over $X\times Y$. Similarly, it has a section over $Y\times X$.

The rest of the argument is similar to those of Theorem 3 and 4. We consider the sets $A=X\times X\cup Y\times Y$
and $B=X\times Y\cup Y\times X$ and argue that both sets admit sections of $\Delta^k_{X\cup Y}$. For example, in the case of $A$ the two sections over $X\times X$ and $Y\times Y$ can only disagree over $C\times C$  which has dimension $\le d$ where $C=X\cap Y$. Since the fibration $\Delta^k_{X\cup Y}$ has $d$-connected fiber those two sections can be joined over $C\times C$ by a fiberwise homotopy. This implies the existence of a section over $A$. The argument in the case of $B$ is similar. 
Next we note that $A\cap B=X\times C\cup Y\times C\cup C\times X\cup C\times Y$ is $d$-dimensional. The same argument implies that there is a continuous section over $A\cup B=(X\cup Y)\times(X\cup Y)$.\qed

\section{Amalgamated product} 

\begin{theorem}\label{amalg}
If all groups $A$, $B$, and $C$ are not counterexamples to the Eilenberg-Ganea conjecture, then
$$
\TC(A\ast_CB)\le\max\{\TC A,\TC B, cd(A\times B), cd(A\ast(C\times \Z)\ast B)+cd(C)+1\}.
$$
\end{theorem}
\begin{proof} Let classifying spaces $BA$, $BB$, and $BC$ for groups $A$, $B$, and $C$ be such that $cd(A)=\dim BA$, $cd(B)=\dim BB$, and $cd(C)=\dim BC$. Here we used that the groups are not potential counterexamples to the Eilenberg-Ganea conjecture.
For the group $G=A\ast_CB$ we consider its classifying space $BG$ of the form of the double mapping cylinder for the maps
$f:BC\to BA$ and $g:BC\to BB$ induced by the amalgamation homomorphisms $C\to A$ and $C\to B$. 
Thus, $BG=X\cup Y$ with $X\cap Y=BC$, $X$ is homotopy equivalent to $BA$ with $\dim X=\max\{cd(A), cd(C)+1\}$, and $Y$ is homotopy equivalent to $BB$ with $\dim Y=\max\{cd(B),cd(C)+1\}$.
Let $$d=\max\{\dim(X\times BC),\dim(Y\times BC)\}=\max\{\dim X,\dim Y\}+cd(C)=$$
$$\max\{cd(A), cd(C)+1,cd(B)\}+cd(C)=cd(A\ast(C\times\Z)\ast B)+cd(C)
.$$
If $$\max\{\TC A,\TC B,\cat(A\times B)\}\ge d+1,$$ then the result follows from Theorem~\ref{th:5}.
Suppose that $$\max\{\TC A,\TC B,\cat(A\times B)\}\le d.$$ We show that $\Delta^{d+1}_{BG}$ admits a section.
Since $\TC A=\TC X\le d+1$ and $\TC B=\TC Y\le d+1$, there are sections over $X\times X$ and $Y\times Y$. Since $\cat(X\times Y)\le d+1$ there are section over $X\times Y$ and $Y\times X$. Since $(X\times X)\cap (Y\times Y)=BC\times BC$ with $\dim(BC\times BC)=2cd(C)\le d$
and the fiber of $\Delta^{d+1}_{BG}$ is $d$-connected, sections over $X\times X$ and $Y\times Y$ can be adjusted over $BC\times BC$ to have a continuous section over $U=(X\times X)\cup (Y\times Y)$. Similarly, we can arrange a continuous section over $V=(X\times Y)\cup(Y\times X)$. Note that $$U\cap V=
(X\times BC)\cup(Y\times BC)\cup (BC\times X)\cup (BC\times Y).$$ Thus $\dim(U\cap V)=d$. Therefore, we can arrange a continuous section over $U\cup V= BG\times BG$.
\end{proof}


\begin{thebibliography}{CLOT}
\bibitem{Ber} I. Berstein, On the Lusternik-Schnirelmann category of Grassmannians. Math. Proc. Cambridge Philos. Soc. 79 (1976), no. 1, 129-134.
\bibitem{Br} K. Brown, {\em Cohomology of groups}, Springer, New York-Heidelberg-Berlin, 1982.
\bibitem{Bre} G. Bredon, {\em Sheaf Theory}. Graduate Text in Mathematics, 170, Springer, New York-Heidelberg-Berlin, 1997.
\bibitem{CV} D. Cohen,  L. Vandembroucq, Topological complexity of the Klein bottle, J. Appl. and Comput. Topology  (2017) https://doi.org/10.1007/s41468-017-0002-0.
\bibitem{CLOT} O. Cornea; G. Lupton; J. Oprea; D. Tanre, {\em Lusternik-Schnirelmann Category}. Mathematical Surveys and Monographs, 103. American Mathematical Society, Providence, RI, 2003.
\bibitem{Dr1} A.~Dranishnikov, Topological complexity of wedges and covering maps, Proc. AMS 142 (2014), 4365-4376.
\bibitem{Dr2} A.~Dranishnikov, On the virtual cohomological dimensions of Coxeter groups. Proc. Amer. Math. Soc. 125 (1997), no. 7, 1885-1891.
\bibitem{Dr3} A.~Dranishnikov, The topological complexity and homotopy cofiber of the diagonal map of non-orientable surfaces, Proc. Amer. Math. Soc. 144 (2016), no 11, 4999-5014.
%\bibitem{Dr14b} A.~Dranishnikov, The LS category of the product of lens spaces, arXiv:1409.8316.
\bibitem{DR} A. Dranishnikov, Yu. Rudyak, On the Berstein-Svarc theorem in dimension 2. Math. Proc. Cambridge Philos. Soc. 146 (2009), no. 2, 407-413.
\bibitem{EG} S. Eilenberg, T. Ganea, On the Lusternik-Schnirelmann Category of Abstract Groups, Annals of
Mathematics, 65, (1957), 517-518.
\bibitem{Fa1} M.~Farber, {\em Invitation to topological robotics}. Zurich Lectures in Advanced Mathematics. European Mathematical Society (EMS), Zurich, 2008. 
\bibitem{Fa2}  M.~Farber,  Topology of robot motion planning. Morse theoretic methods in nonlinear analysis and in symplectic topology, 185-230, NATO Sci. Ser. II Math. Phys. Chem., 217, Springer, Dordrecht, 2006.
\bibitem{FM} M.~Farber, S~Mescher, On the topological complexity of aspherical spaces, Preprint arXiv:1708.06732v2 [math.AT].
\bibitem{Gr} M.~Grant,  Topological complexity, fibrations and symmetry, Topol. Appl. 159 (2012), no 1, 88-97.
\bibitem{I-S}  N. Iwase, M. Sakai,  Topological complexity is a fibrewise L-S category. Topology Appl. 157 (2010), no. 1, 10-21.
\bibitem{Ru} Yu. Rudyak, On topological complexity of Eilenberg-MacLane spaces. 
Topology Proc. 48 (2016), 65-67.
\bibitem{Sch}  A.~Schwarz, The genus of a fibered space. Trudy Moscov. Mat. Obsc. 10, 11 (1961 and 1962), 217-272, 99-126.
\bibitem{St} J. Stallings, On torsion-free groups with infinitely many ends, Ann. of Math. 88 (1968), 312-334. 
\bibitem{Sw} R. Swan, Groups of cohomological dimension one, J. Algebra 12 (1969), 585-610.
\bibitem{TD} Tammo tom Dieck, {\em Algebraic Topology}. European Mathematical Society, 2008.
\end{thebibliography}
\end{document}